\documentclass[reqno]{amsart}
\usepackage[scale=0.75, centering, headheight=14pt]{geometry}
\usepackage[latin1]{inputenc}
\usepackage[T1]{fontenc}
\usepackage{lmodern}
\usepackage[english]{babel}
\usepackage{graphicx} 

\usepackage{hyperref}

\hypersetup{
    colorlinks = true,
    citecolor = blue
}

\usepackage{mathrsfs, dsfont}
\usepackage{amssymb}
\usepackage{amsthm}
\usepackage{amsmath,amsfonts,amssymb,esint}
\usepackage{graphics,color}
\usepackage{enumerate}
\usepackage{mathtools,centernot}
\usepackage{tikz}
\usetikzlibrary{patterns} 
\usetikzlibrary{decorations.pathmorphing}
\usetikzlibrary{arrows.meta, positioning}
\usepackage{cleveref}
\usepackage{subcaption}
\usepackage{aliascnt}
\usepackage[initials, alphabetic, nobysame]{amsrefs}
\usepackage{soul}

\makeatletter
\def\newaliasedtheorem#1[#2]#3{
  \newaliascnt{#1@alt}{#2}
  \newtheorem{#1}[#1@alt]{#3}
  \expandafter\newcommand\csname #1@altname\endcsname{#3}
}
\makeatother

\usepackage{indentfirst}

\usepackage{esint}
\usepackage{mathrsfs}
\usepackage{stmaryrd}

\makeatletter
\newsavebox{\measure@tikzpicture}

\newcommand{\setword}[2]{%
  \phantomsection
  #1\def\@currentlabel{\unexpanded{#1}}\label{#2}%
}

\renewcommand\labelenumi{(\roman{enumi})}
\renewcommand\theenumi\labelenumi

\newtheorem{theorem}{\bf Theorem}[section]
\newtheorem{remark}[theorem]{\bf{Remark}}
\newtheorem{lemma}[theorem]{\bf Lemma}

\newtheorem*{theorem*}{Theorem}
\newtheorem*{proposition*}{Proposition}
\newtheorem{claim}[theorem]{\bf{Claim}}
\theoremstyle{definition}
\newtheorem{example}[theorem]{Example}
\newtheorem{definition}[theorem]{\bf Definition}

\newcommand{\Image}{\operatorname{Im}}

\newcommand{\eps}{\varepsilon}
\newcommand{\R}{\mathbb R}
\newcommand{\N}{\mathbb N}

\newcommand{\Z}{\mathbb Z}

\newcommand{\T}{\mathbb T}

\newcommand{\Leb}{\mathcal{L}}

\newcommand{\metric}{\mathsf{d}}
\newcommand{\divergence}{\operatorname{div}}

\newcommand{\topo}{\operatorname{top}}
\newcommand{\per}{\operatorname{per}}
\newcommand{\ClassOfMaps}{\mathcal{F}}

\DeclareMathOperator{\Homeo}{Homeo}

\DeclareMathOperator{\supp}{supp}
\DeclareMathOperator{\argmin}{argmin}

\allowdisplaybreaks
\numberwithin{equation}{section}

\author[Carl J. P. Johansson]{Carl Johan Peter Johansson}
\address{Institute of Mathematics, EPFL, Station 8, 1015 Lausanne, Switzerland}
\email{carl.johansson@epfl.ch}
\author[Giulia Mescolini]{Giulia Mescolini}
\address{Institute of Mathematics, EPFL, Station 8, 1015 Lausanne, Switzerland}
\email{giulia.mescolini@epfl.ch}

\subjclass[2020]{37C20, 37B40, 35Q49}
\keywords{Topological entropy, Flow maps, Time-periodic velocity fields, Osgood continuity}

\title[Infinite topological entropy for velocity fields]{Topological entropy is generically infinite for non-Lipschitz velocity fields}

\begin{document}

\begin{abstract}
 We prove that for any Osgood non-Lipschitz modulus of continuity $\omega$, flow maps associated with time-periodic $\omega$-continuous velocity fields generically (in the sense of Baire) have infinite topological entropy.
\end{abstract}

\maketitle

\section{Introduction}
In this note, we study topological entropy of flow maps
\begin{equation}
\label{eq:flowmap}
\left\{
\begin{array}{l}
    \frac{d}{dt} X^{b}_t(x) = b(t, X_t^{b}(x)); \\
    X_0^b (x) = x;
\end{array}
\right.
\end{equation}
associated with time-periodic velocity fields $b \colon \R_+ \times \T^2 \to \R^2$ on the two-dimensional flat torus. Without loss of generality, the period is assumed to be one, i.e. $b(t+1, x) = b(t,x)$ for all $t \in \R_+$ and $x \in \T^2$. Topological entropy is a central notion in the modern theory of dynamical systems and it measures the rate of increase of topological complexity of a dynamical system as it evolves over time (see Section~\ref{sec:TopologicalEntropy}). If $b$ is Lipschitz in space, then $X_1^b$ is also Lipschitz and hence its topological entropy is finite, see \cite{SI70, RB71}.
For continuous and H\"older homeomorphisms, it has been proved that topological entropy is generically infinite by Yano \cite{KY80} and De Faria, Hazard and Tresser \cite{EDFPHCT21}, respectively. 
To the best of the authors' knowledge, no results addressing the question of the typical behavior of topological entropy for flow maps associated to velocity fields below Lipschitz regularity are available in the literature. The goal of this note is to pursue that question.  
We will consider non-Lipschitz Osgood moduli of continuity, i.e. $\omega \in C^1(\R_+;\R_+)$ increasing, concave and such that
\begin{equation}
\label{eq:osgood_hp}
\omega(0) = 0, \quad \int_0^1 \dfrac{1}{\omega(s)} \, ds = \infty \quad \text{and} \quad \lim_{s \to 0} \dfrac{s}{\omega(s)} = 0.
\end{equation}
An example is the $\log$-Lipschitz modulus $\omega_{LL}$, namely $\omega_{LL}(s) = - s \log s$ near $0$.
We consider the space of $1$-periodic velocity fields which are integrable in time and $\omega$-continuous in space with vanishing $\omega$-oscillation, denoted by $L^1_{\per}(\R_+ ; c_{\omega}(\T^2; \R^2))$ (see Subsection~\ref{subsec:Notation}).
The Osgood assumption guarantees uniqueness and continuity of the flow map \cite{WFO98}. The 1-periodicity implies $X_n^b = (X_1^b)^n$ for all $b \in L^1_{\per}c_{\omega}$ and $n \in \N_{\geq 1}$, meaning that it behaves as a discrete dynamical system, making topological entropy a meaningful quantity to describe the topological complexity of $X_t^b$.
We prove that flow maps associated with velocity fields in $L^1_{\per}(\R_+ ; c_{\omega}(\T^2; \R^2))$ generically (in the sense of Baire) have infinite topological entropy:
\begin{theorem}
\label{thm:main}
    Let $\omega \in C^1(\R_+;\R_+)$ be a non-Lipschitz Osgood modulus. Then {$L^1_{\per}(\R_+; c_{\omega}(\T^2; \R^2))$} contains a residual set $\mathcal{X}$ such that for all $b \in \mathcal{X}$, the topological entropy of $X_1^b$ is infinite.
\end{theorem}

The proof (carried out in Section~\ref{sec:ProofOfMainThm}) relies on perturbing any velocity field in $L^1_{\per} c_{\omega}$ to find a periodic point around whose orbit we perturb further to increase entropy. This last perturbation, which is constructed in Section~\ref{sec:construction_pseudohorseshoe} and builds on the idea of pseudo-horseshoes of Yano \cite{KY80}, is structurally stable which leads to genericity.

\begin{remark}
    Theorem~\ref{thm:main} is stated in two dimensions but can be generalized to any dimension $\geq 2$.
\end{remark}

\begin{remark}
In two dimensions, if $b$ is Osgood and autonomous, then $h_{\topo}(X_1^b) = 0$ due to a result of Young \cite{LSY77}.
\end{remark}

\begin{remark}
Theorem~\ref{thm:main} also holds with $L^1_{\per}(\R_+ ; c_{\omega}(\T^2; \R^2))$ replaced by $L^1_{\per}(\R_+ ; c_{\omega}(\T^2; \R^2)) \cap \{ \divergence b \in L^{\infty}\}$ or $L^1_{\per}(\R_+ ; c_{\omega}(\T^2; \R^2)) \cap \{ \divergence b = 0\}$. For the case of bounded divergence, the proof remains identical while the divergence-free case requires a small adaptation in the proof of Claim~\ref{claim:FirstPerturbation}, see Remark~\ref{rmk:divfree}.
\end{remark}

In recent years, techniques from both deterministic and stochastic dynamical systems theory, applied to problems in fluid dynamics, have led to significant progress in mixing \cite{TMEAZ19, SBMZ22, MZ22, ABMCZRG23, EBMCCJ25}, enhanced dissipation \cite{TMEKLJM25, JBABSPS21} and passive scalar turbulence \cite{JBABSPS22B, ABKKH24, WCKR25}. In \cite{EBMCCJ25}, it was proved that metric entropy of flow maps (regular in the sense of DiPerna-Lions-Ambrosio \cite{DPL89, A04}) associated with incompressible $W^{1,1+\eps}$ Sobolev velocity fields is finite and bounded by $\| \nabla b \|_{L^1_{t,x}}$. While the Sobolev assumption is natural for metric entropy, topological entropy requires continuity of the flow map, guaranteed by the Osgood assumption in Theorem~\ref{thm:main}.
Topological entropy is related to metric entropy via the variational principle for entropy \cite[Theorem 10.1]{MV16}:
\begin{equation}
h_{\topo}(T) =  \sup_{\mu : T_{\#} \mu = \mu } h_{\mu}(T),
\end{equation}
where $h_{\mu}(T)$ denotes the metric entropy of $T$ with respect to an invariant measure $\mu$.
Hence, Theorem~\ref{thm:main} is evidence of super-exponential growth of complexity occurring in a (possibly low-dimensional) subset of $\T^2$. This is consistent with super-exponential loss of regularity, which may occur in Osgood velocity fields.
Finally, besides being a sufficient (and in fact necessary as shown by \cite{RCAK26}) condition for uniqueness and continuity of the flow map, Osgood regularity is a compelling regularity class to study since it holds in several important applications where the velocity fails to be Lipschitz. Indeed, in the vorticity formulation of the two-dimensional Euler equations with initial vorticity in $L^{\infty}$, the velocity field is log-Lipschitz continuous \cite{VIY63} but fails to be Lipschitz as shown by the fact that the gradient of vorticity may grow double-exponentially \cite{AKVS14}. 
Note that our method also allows to consider velocity fields with exponentially integrable gradient (see Example~\ref{rmk:ExpIntGradient}), a property of such Euler velocity fields \cite[Lemma 3.2]{EBQHN21}.
See also \cite{HBJYC94, EBQHN21, TDDTMEJL23, DNMS24} and the references therein for further results regarding regularity in the two-dimensional Euler equations with bounded vorticity. 
In addition, velocity fields with Osgood modulus of continuity appear in various kinetic models, see for instance \cite{GCMICSGS24, JJAR25}.

\subsection{Notation}\label{subsec:Notation}
We work on the $2$-dimensional flat torus $\T^2$ with the geodesic distance $\metric$. 
When convenient, we identify $\T^2$ with $[-1,1]^2$. 
A constant $c$ depending on parameters $\ast$ is denoted by $c(\ast)$.
For any non-Lipschitz Osgood modulus of continuity, we write $[f]_{\omega} \coloneqq \sup_{x \neq y} \frac{|f(x) - f(y)|}{\omega(\metric(x,y))}$ and $\| f \|_{\omega} \coloneqq \| f \|_{C^0} + [f]_{\omega}$.
Recall $C_{\omega} \coloneqq \{ f : \| f \|_{\omega} < \infty \}$ and denote the closure of $C^{\infty}$ with respect to $\| f \|_{\omega}$ as $c_{\omega}$.
Note that $f \in C_{\omega}$ belongs to $c_{\omega}$ if and only if it has vanishing $\omega$-oscillation, i.e. $\lim_{r \to 0} \sup_{\metric(x,y)<r} \frac{|f(x) - f(y)|}{\omega(\metric(x,y))} = 0$.
Thus $c_{\omega} \subsetneq C_{\omega}$ and $C_{\overline{\omega}} \subsetneq c_{\omega}$ if $\lim_{s \to 0} \frac{\overline{\omega}(s)}{\omega(s)} = 0$. We write
\[
 L^1_{\per}(\R_+; c_{\omega}(\T^2; \R^2)) \coloneqq \left\{ b : \| b \|_{L^1_{\per}c_{\omega}} \coloneqq \int_0^1 \| b(t, \cdot) \|_{\omega} \, dt < \infty, \ b(t,\cdot) = b(t+1, \cdot) \, \forall t \in \R_+ \right\}.
\]
When convenient, we identify $1$-periodic functions with their restriction to $[0,1]$. 

\subsection*{Acknowledgments} The authors were supported by the Swiss State Secretariat for Education, Research and Innovation (SERI) under contract number MB22.00034 through the project TENSE. The authors would like to thank Maria Colombo and Elia Bru\'e for their guidance and valuable suggestions. The authors thank David Villringer for pointing out reference \cite{LSY77}.

\section{Topological entropy}\label{sec:TopologicalEntropy}
 We recall the notion of topological entropy. For a more detailed treatment of the topic, we refer to \cite{MR648108,MBGS02,MV16}.
 Let $(X, \metric)$ be a compact metric space and consider continuous maps $T \colon X \to X$. A collection of open sets $\alpha$ is called an \emph{open cover} if its union equals $X$. For any open covers $\alpha_1, \ldots, \alpha_n$, we define their \emph{join} as 
 \[
  \bigvee_{i = 1}^n \alpha_i = \left\{ \bigcap_{i = 1}^n A_i : A_i \in \alpha_i , i = 1, \ldots, n \right\}.
 \]
 \begin{definition}[Topological entropy]
     For any open cover $\alpha$, we define $N(\alpha)$ as the cardinality of the smallest subcover. For any continuous $T \colon X \to X$ and any open cover $\alpha$, we set
     \[
      h_{\topo}(T, \alpha) = \lim_{n \to \infty} \dfrac{1}{n} \log N \left( \bigvee_{i = 0}^{n-1} T^{- i} \alpha \right)
     \]
     where $T^{-i} \alpha = \{ T^{-i} O : O \in \alpha \}$ denotes the \emph{pull-back open cover}.
     Then the \emph{topological entropy} of $T$ is defined by
     \[
      h_{\topo}(T) = \sup_{\alpha} h_{\topo}(T, \alpha)
     \]
     where the supremum is taken over all finite open covers $\alpha$. 
 \end{definition}
 As an example, consider $\Sigma_k = \{ 0, \ldots, k-1 \}^{\Z}$ and the shift map $\sigma \colon \Sigma_k \to \Sigma_k$ defined by $(\sigma(w))_i = w_{i-1}$. Its topological entropy equals $\log k$ by \cite[Theorem 7.11]{MR648108}. In particular, it is straightforward to check that $h_{\topo}(\sigma) \geq \log k$ by taking $\alpha = \{ A_0, \ldots, A_{k-1} \}$ with $A_i = \{ w \in \Sigma_k : w_0 = i \}$. Then, 
 \[
  N \left( \bigvee_{i = 0}^{n-1} \sigma^{- i} \alpha \right) = k^n
 \]
 and hence $h_{\topo}(\sigma) \geq h_{\topo}(\sigma, \alpha) = \log k$. 
 \begin{definition}[Factor]
Let $T \colon X \to X$, $S \colon Y \to Y$ be continuous maps. A \emph{semiconjugacy} from $T$ to $S$ is a continuous surjective map $\pi \colon X \to Y$ such that $S \circ \pi = \pi \circ T$. If there is a semiconjugacy $\pi$ from $T$ to $S$, then $S$ is called a \emph{factor} of $T$.
\end{definition}

\begin{lemma}[{\cite[Proposition 2.5.6]{MBGS02}}]\label{lemma:EntropyOfFactors}
    Let $S \colon Y \to Y$ be a factor of $T \colon X \to X$, then $h_{\topo}(T) \geq h_{\topo}(S)$.
\end{lemma}

\begin{lemma}[{\cite[Proposition 2.5.5]{MBGS02}}]\label{lemma:EntropyOnInvariantSubsets}
    Let $T \colon X \to X$ be a continuous map. If $A_i$, $i = 1, \ldots, k$ are closed subsets of $X$ such that $T(A_i) = A_i$ and whose union is $X$, then
    \[
     h_{\topo}(T) = \max_{1 \leq i \leq k} h_{\topo}(T |_{A_i}).
    \]
\end{lemma}

 \begin{lemma}[{\cite[Proposition 2.5.5]{MBGS02}}]\label{lemma:EntropyOfIterations}
    Let $T \colon X \to X$ be a continuous map. Then $h_{\topo}(T^n) = n h_{\topo}(T)$ for all integers $n \geq 1$.
\end{lemma}

\section{Entropy-increasing building block} 
\label{sec:construction_pseudohorseshoe}
We build a velocity field supported in a small cylinder such that its time-one flow map has large entropy.
\begin{lemma}
\label{lemma:pert_horseshoe_construction}
    Let $N \geq 1$ be an integer. There exists a collection of smooth incompressible velocity fields $\{b_{\eps} \}_{\eps \in (0, 1]}$ from $[0, 1] \times \T^2$ to $\R^2$ such that
    \begin{itemize}
        \item $\lim_{\eps \to 0} \| b_{\eps} \|_{C^0([0,1];c_\omega)} = 0$;
        \item $\supp (b_{\eps}) \subseteq [0,1] \times B_{\eps}(0)$;
        \item for any $\eps \in (0, 1]$, there exists $\delta = \delta(N, \eps) > 0$ such that if a homeomorphism $T \in \Homeo(\T^2)$ satisfies
    \end{itemize}
    \begin{equation}\label{eq:ConditionToPreserveEntropy}
     \sup_{x \in B_{\eps}(0)} \metric ( T(x) , X_1^{b_{\eps}}(x) ) < \delta,
    \end{equation}
    then $h_{\topo}(T) \geq \log N$.
\end{lemma}

\begin{proof}
    For all $i = 0, \ldots, N$, we define $E_i = \left\{ - \frac{1}{4} + \frac{i}{2 N} \right\} \times \left[- \frac{1}{4}, \frac{1}{4}\right]$, $D_{-1} = \left(- \frac{1}{2}, - \frac{1}{3}\right) \times \left(- \frac{1}{4}, \frac{1}{4}\right)$ and {$D_{1}~=~\left(\frac{1}{3},\frac{1}{2}\right)~\times~\left(- \frac{1}{4}, \frac{1}{4}\right)$}, see Figure~\ref{fig:beforeT} for a depiction of these sets.
\begin{figure}[htbp]
\centering

\begin{subfigure}{0.48\textwidth}
\centering
\begin{tikzpicture}[scale=4]
    \fill[pink!70] (-0.25,-0.25) rectangle (0.25,0.25);

    \fill[black!30] (-0.5,-0.25) rectangle (-0.3333,0.25);
    \node[text=black!50, below] at (-0.4167,-0.25) {\footnotesize{$D_{-1}$}};

    \fill[black!30] (0.3333,-0.25) rectangle (0.5,0.25);
    \node[text=black!50, below] at (0.4167,-0.25) {\footnotesize{$D_{1}$}};

    \fill[black] (-0.25,-0.25) rectangle (-0.125,0.25);
    \fill[black] (-0.125,-0.25) rectangle (0,0.25);
    \fill[black] (0,-0.25) rectangle (0.125,0.25);
    \fill[black] (0.125,-0.25) rectangle (0.25,0.25);

    \draw[dashed] (-0.7,0.25) -- (0.7,0.25);
    \draw[dashed] (-0.7,-0.25) -- (0.7,-0.25);
    \draw[dashed] (0.5,-0.5) -- (0.5,0.5);
    \draw[dashed] (-0.5,-0.5) -- (-0.5,0.5);

    \draw[black!50,thick] (-0.25,-0.25) -- (-0.25,0.25);
    \draw[black!50,thick] (-0.125,-0.25) -- (-0.125,0.25);
    \draw[black!50,thick] (0,-0.25) -- (0,0.25);
    \draw[black!50,thick] (0.125,-0.25) -- (0.125,0.25);
    \draw[black!50,thick] (0.25,-0.25) -- (0.25,0.25);

    \node[black!70,above] at (-0.25,0.25) {\footnotesize{$E_0$}};
    \node[black!70,above] at (-0.125,0.25) {\footnotesize{$E_1$}};
    \node[black!70,above] at (0,0.25) {\footnotesize{$E_2$}};
    \node[black!70,above] at (0.125,0.25) {\footnotesize{$E_3$}};
    \node[black!70,above] at (0.25,0.25) {\footnotesize{$E_4$}};
    \node[black,below] at (-0.7,0.25) {\footnotesize{$1/4$}};
    \node[black,below] at (-0.7,-0.25) {\footnotesize{$-1/4$}};
    \node[black,below] at (-0.5,-0.5) {\footnotesize{$-1/2$}};
    \node[black,below] at (0.5,-0.5) {\footnotesize{$1/2$}};
    \node[white, below right] at (-0.25,0.05) {\footnotesize{$F_0$}};
    \node[white, below right] at (-0.125,0.05) {\footnotesize{$F_1$}};
    \node[white, below right] at (0,0.05) {\footnotesize{$F_2$}};
    \node[white, below right] at (0.125,0.05) {\footnotesize{$F_3$}};
\end{tikzpicture}
\caption{The sets $E_i$ and $F_i$.}
\label{fig:beforeT}
\end{subfigure}
\hfill
\begin{subfigure}{0.48\textwidth}
\centering
\begin{tikzpicture}[scale=4]
    \fill[black!30] (-0.5,-0.25) rectangle (-0.3333,0.25);
    \node[text=black!50, below] at (-0.4167,-0.25) {\footnotesize{$D_{-1}$}};

    \fill[black!30] (0.3333,-0.25) rectangle (0.5,0.25);
    \node[text=black!50, below] at (0.4167,-0.25) {\footnotesize{$D_{1}$}};

    \draw[black,thick,decorate,decoration={snake,amplitude=0.3,segment length=2mm}] (-0.37,0.15) -- (0.37,0.15);
    \draw[black,thick,decorate,decoration={snake,amplitude=0.3,segment length=2mm}] (0.37,0.23) -- (-0.37,0.23);
    \draw[black!70,thick,decorate,decoration={snake,amplitude=0.3,segment length=2mm}] (-0.37,0.23) -- (-0.37,0.15);

    \draw[black,thick,decorate,decoration={snake,amplitude=0.3,segment length=2mm}] (-0.37,0.02) -- (0.37,0.02);
    \draw[black,thick,decorate,decoration={snake,amplitude=0.3,segment length=2mm}] (0.37,0.1) -- (-0.37,0.1);

    \draw[black,thick,decorate,decoration={snake,amplitude=0.3,segment length=2mm}] (-0.37,-0.01) -- (0.37,-0.01);
    \draw[black,thick,decorate,decoration={snake,amplitude=0.3,segment length=2mm}] (0.37,-0.1) -- (-0.37,-0.1);

    \draw[black,thick,decorate,decoration={snake,amplitude=0.3,segment length=2mm}] (-0.37,-0.23) -- (0.37,-0.23);
    \draw[black,thick,decorate,decoration={snake,amplitude=0.3,segment length=2mm}] (0.37,-0.15) -- (-0.37,-0.15);
    \draw[black!70,thick,decorate,decoration={snake,amplitude=0.3,segment length=2mm}] (-0.37,-0.23) -- (-0.37,-0.15);

    \draw[black!70,thick,decorate,decoration={snake,amplitude=0.1,segment length=2mm}] (0.37,-0.15) -- (0.45,-0.10);
    \draw[black!70,thick,decorate,decoration={snake,amplitude=0.1,segment length=2mm}] (-0.4,0) -- (-0.48,0);
    \draw[black!70,thick,decorate,decoration={snake,amplitude=0.1,segment length=2mm}] (0.37,0.15) -- (0.46,0.16);

    \fill[color=black] (-0.25,0.15) rectangle (0.25,0.23);
    \fill[color=black] (-0.25,0.02) rectangle (0.25,0.1);
    \fill[color=black] (-0.25,-0.1) rectangle (0.25,-0.01);
    \fill[color=black] (-0.25,-0.23) rectangle (0.25,-0.15);

    \node[white] at (0,-0.2) {\footnotesize{$H_0$}};
    \node[white] at (0,-0.05) {\footnotesize{$H_1$}};
    \node[white] at (0,0.055) {\footnotesize{$H_2$}};
    \node[white] at (0,0.195) {\footnotesize{$H_3$}};

    \draw[black, thick, decorate, decoration={snake,amplitude=0.3,segment length=2mm}]
  (0.37,0.23)
  .. controls (0.5,0.125) and (0.5,0.125) ..
  (0.37,0.02);

    \draw[black, thick, decorate, decoration={snake,amplitude=0.3,segment length=2mm}]
  (0.37,0.15)
  .. controls (0.39,0.125) and (0.39,0.125) ..
  (0.37,0.1);

    \draw[black, thick, decorate, decoration={snake,amplitude=0.3,segment length=2mm}]
  (-0.37,-0.01)
  .. controls (-0.41,0) and (-0.41,0) ..
  (-0.37,0.02);

    \draw[black, thick, decorate, decoration={snake,amplitude=0.3,segment length=2mm}]
  (-0.37,0.1)
  .. controls (-0.5,0) and (-0.5,0) ..
  (-0.37,-0.1);

    \draw[black, thick, decorate, decoration={snake,amplitude=0.3,segment length=2mm}]
  (0.37,-0.01)
  .. controls (0.5,-0.15) and (0.5,-0.15) ..
  (0.37,-0.23);

    \draw[black, thick, decorate, decoration={snake,amplitude=0.3,segment length=2mm}]
  (0.37,-0.1)
  .. controls (0.39,-0.125) and (0.39,-0.125) ..
  (0.37,-0.15);

    \draw[dashed] (-0.7,0.25) -- (0.7,0.25);
    \draw[dashed] (-0.7,-0.25) -- (0.7,-0.25);
    \draw[dashed] (0.5,-0.5) -- (0.5,0.5);
    \draw[dashed] (-0.5,-0.5) -- (-0.5,0.5);

    \node[black,below] at (-0.7,0.25) {\footnotesize{$1/4$}};
    \node[black,below] at (-0.7,-0.25) {\footnotesize{$-1/4$}};
    \node[black,below] at (-0.5,-0.5) {\footnotesize{$-1/2$}};
    \node[black,below] at (0.5,-0.5) {\footnotesize{$1/2$}};

    \node[black!70, right] at (-0.45,0.32) {\footnotesize{$T(E_4)$}};
    \node[black!70, right] at (0.49,0.18) {\footnotesize{$T(E_3)$}};
    \node[black!70, right] at (-0.75,0) {\footnotesize{$T(E_2)$}};
    \node[black!70, right] at (0.49,-0.12) {\footnotesize{$T(E_1)$}};
    \node[black!70, right] at (-0.75,-0.2) {\footnotesize{$T(E_0)$}};
\end{tikzpicture}
\caption{The sets $T(E_i)$ and $H_i$. For a related illustration, see \cite[Figure 1]{KY80}.}
\label{fig:afterT}
\end{subfigure}

\caption{Comparison of subsets before and after transformation $T$ with $N=4$.}
\label{fig:combined}
\end{figure}
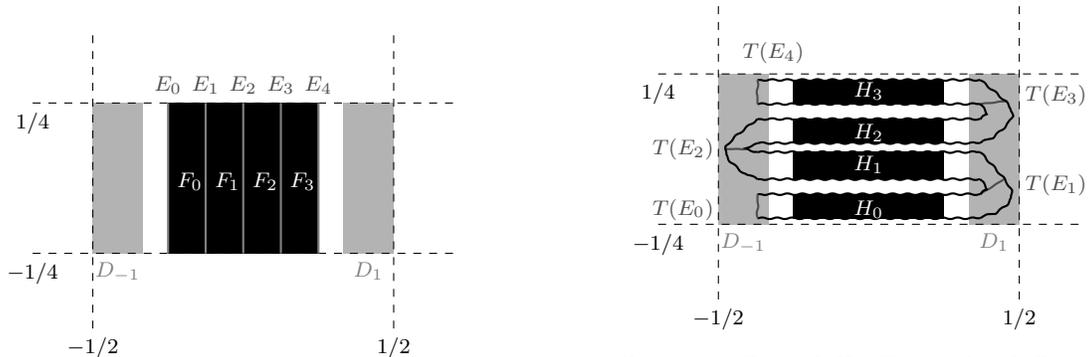
    We say that a homeomorphism $T \in \Homeo(\T^2)$ belongs to $\ClassOfMaps$ if 
    \begin{enumerate}
        \item $T([- 1/4, 1/4]^2) \subseteq (- 1/2, 1/2) \times (- 1/4, 1/4)$; \label{item:PropertyOfFMapsOne}
        \item $T(E_i) \subseteq D_{(-1)^{i}}$ for all $i = 0, \ldots, N$. \label{item:PropertyOfFMapsTwo}
    \end{enumerate}
    It is clear that there exists a smooth divergence-free (stretching and folding velocity field) $b \colon [0,1] \times \T^2 \to \R^2$ with $\supp(b) \subseteq [0,1] \times B_1(0)$ such that $X^b_1 \in \ClassOfMaps$. Then by defining $b_{\eps} \colon [0,1] \times B_{\eps}(0) \to \R^2$ as $b_{\eps}(t,x) = \eps b \left( t, {x}/{\eps} \right)$ and extending it by $0$ outside $B_{\eps}(0)$, we get $\supp (b_{\eps}) \subseteq [0,1] \times B_{\eps}(0)$. Additionally, 
    \[
     \sup_{x,y \in \T^2} \frac{|b_{\eps}(t,x) - b_{\eps}(t,y)|}{\metric(x,y)} \leq \| \nabla b(t,\cdot) \|_{C^0}, \quad  \sup_{x,y \in B_{\eps}(0)} \frac{\metric(x,y)}{\omega(\metric(x,y))} \to 0 \quad \text{as $\eps \to 0$,}
    \]
    from which we deduce that $\lim_{\eps \to 0} \| b_{\eps} \|_{C^0([0,1];c_\omega)} = 0$.
    We now show that there exists $\delta = \delta(\eps)$ such that if $T \in \Homeo(\T^2)$ satisfies \eqref{eq:ConditionToPreserveEntropy}, then $h_{\topo}(T) \geq \log N$. Since, up to a rescaling, the proof is identical for all $\eps > 0$, without loss of generality we assume $\eps = 1$. Recall the definition of $\ClassOfMaps$ and note that for any $T \in \ClassOfMaps$, there exists $\rho$ such that for any $\tilde{T} \in \Homeo(\T^2)$, $\sup_{x \in B_1(0)} \metric (T(x), \tilde{T}(x)) < \rho$ implies $\tilde{T} \in \ClassOfMaps$. Therefore, since $X_1^b \in \ClassOfMaps$, it is enough to show that for all $T \in \ClassOfMaps$ we have $h_{\topo}(T) \geq \log N$.
    We follow the strategy of \cite[Proposition 2]{KY80}.
     We define the sets
    \[
     F_i = \left[ - \frac{1}{4} + \frac{i}{2 N} , - \frac{1}{4} + \frac{i+1}{2 N} \right] \times \left[- \frac{1}{4}, \frac{1}{4}\right] \quad i = 0, \ldots, N-1
    \]
    and $H_i = [- 1/4, 1/4]^2 \cap T(F_i)$, $i = 0, \ldots, N-1$ as in Figure~\ref{fig:afterT}.
    We set $H = \cup_{i = 0}^{N-1} H_i$ and $\Lambda = \cap_{k \in \Z} T^k(H)$.
    We recall that $\Sigma_N$ denotes $\{ 0, \ldots, N-1\}^{\Z}$ and $\sigma \colon \Sigma_N \to \Sigma_N$ denotes the left shift map, namely $(\sigma(x))_i = x_{i+1}$ for any $i \in \Z$. We define $\pi \colon \Lambda \to \Sigma_N$ by
    \[
     (\pi(x))_j = i \quad \text{if $T^j(x) \in H_i$.}
    \]
    This map is continuous since the sets $H_i$ are mutually disjoint and compact.
    \begin{claim}
    The map $\pi$ is surjective. 
    \end{claim}
    \begin{proof}[Proof of Claim]
    We need to prove that for any $\{a_j\}_{j\in \mathbb Z}$, there exists $x \in \Lambda$ such that $T^j(x) \in H_{a_j}$. To this end, we prove that the intersection $\bigcap_{j\in\mathbb Z} T^{-j}(H_{a_j})$ is non-empty. By induction, we prove that for any finite sequence of indices $\{a_{0},\dots, a_k\}$ there exists a continuous curve $\gamma: [0,1] \to \T^2$ such that
    \begin{equation}
        \Image ( \gamma ) \subseteq H_{a_{0}}\cap \dots \cap T^k(H_{a_{k}}), \quad \gamma(0) \in E_0, \quad \gamma(1) \in E_N.
    \end{equation}
    Indeed,
    \begin{itemize}
        \item $k=0$: for any $a_0 = 0,\dots,N-1$, it follows from the definition of $H_i$ that such a curve exists.
        \item $k\geq 1$: by induction, assume that there exists a continuous $\gamma:[0,1] \to \T^2$ such that $\Image ( \gamma ) \subseteq
    H_{a_{1}} \cap \ldots \cap T^{k-1}(H_{a_{k}})$, $\gamma(0) \in E_0$ and $\gamma(1) \in E_N$.
        Then, we have 
        $\Image( T(\gamma) ) \subset T(H_{a_{1}}) \cap \ldots \cap T^{k}(H_{a_k})$ and since $\gamma$ intersects all the sets $E_i$, by the definition of $H_i$ there exist $0 \leq s_1 < s_2 \leq 1$ such that $T(\gamma(s_1)) \in E_0$, $T(\gamma(s_2)) \in E_N$ and $T(\gamma([s_1, s_2])) \subseteq H_{a_0}$. Letting $\tilde{\gamma}$ be the restriction of $T(\gamma)$ to $[s_1, s_2]$, we get 
        \begin{equation}
    \Image(\tilde{\gamma}) \subseteq H_{a_{0}} \cap \dots \cap T^k(H_{a_{k}}).
    \end{equation}
    \end{itemize}
    Since any finite intersection $H_{a_{0}}\cap \ldots \cap T^k(H_{a_{k}})$ contains a continuous curve, it is non-empty. Since $T \in \Homeo(\T^2)$, by taking pre-images, it is clear that any finite intersection of the form $T^{-k}(H_{a_{-k}})\cap \ldots \cap T^k(H_{a_{k}})$ is also non-empty.
    Finally, since $T \in \Homeo(\T^2)$ and $H_i$ is a closed set for any $i=0,\dots,N-1$, the sets $T^k(H_{a_k})$ are compact. Hence, by the finite intersection property applied to the non-empty finite intersections, we deduce that the intersection $\bigcap_{j\in\mathbb Z} T^{-j}(H_{a_j})$ is non-empty.
    \end{proof}
     To conclude the proof, since $\pi$ is continuous surjective and $\pi \circ T = \sigma \circ \pi$, $\sigma$ is a factor of $T |_{\Lambda}$. Moreover, $h_{\topo}(\sigma) = \log N$ by \cite[Theorem 7.11]{MR648108} and since $\sigma$ is a factor of $T |_{\Lambda}$, we deduce that $h_{\topo}(T |_{\Lambda}) \geq \log N$ by Lemma~\ref{lemma:EntropyOfFactors}. By Lemma~\ref{lemma:EntropyOnInvariantSubsets}, $h_{\topo}(T) \geq \log N$. 
\end{proof}

\begin{example}
\label{example:vel_field}
    By Lemma~\ref{lemma:pert_horseshoe_construction}, it is straightforward that there exists a velocity field in $L^1_{\per}(\R_+; c_\omega(\T^2; \R^2))$ such that its time-one flow map has infinite topological entropy.
    Indeed, for any  $N \in \mathbb N$, by  Lemma~\ref{lemma:pert_horseshoe_construction} there exists $b_N$ such that $\|b_N\|_{C^0([0,1];c_\omega)} \leq N^{-2}$, $\supp(b_N) \subseteq [0,1]\times B_{2^{-N-2}}(0)$ and its time-one flow map $X^{b_N}_1$  
    satisfies $h_{\topo} (X^{b_N}_1) \geq N$.
    
    Now, let $x_N = (-1/2, -1+2^{-N+1})$ and define $b \colon [0,1] \times \mathbb T^2 \to \mathbb R^2$ as $b(t,x) \coloneqq \sum_{N \in \mathbb N} b_N(t, x-x_N)$. Let $X^{b}_1$ be the time-one flow map of $b$ and observe that, since the velocity fields $b_N(\cdot - x_N)$ have mutually disjoint supports, $X^{b}_1 = x_N + X_1^{b_N}(\cdot - x_N)$ in $B_{2^{-N-2}}(x_N)$. For the same reason, $\overline{B_{2^{-N-2}}(x_N)}$ is invariant under $X_1^b$ and therefore Lemma~\ref{lemma:EntropyOnInvariantSubsets} implies $h_{\topo} (X^b_1) = +\infty$.
    We are left to prove that $b$ has the desired regularity; to do so, observe that 
\begin{equation}
    \|b(t,\cdot)\|_{\omega} \leq \sum_{N\in \mathbb N} \|b_N(t, \cdot)\|_{\omega} \leq \sum_{N\in \mathbb N} N^{-2} \leq C < \infty \quad \forall t \geq 0.
\end{equation}
\end{example}

\begin{example}\label{rmk:ExpIntGradient}
   Using the technique of Example~\ref{example:vel_field}, we can build an example of a velocity field $u$ whose time-one flow map has infinite topological entropy and whose gradient is exponentially integrable, satisfying
   \begin{equation}
   \label{eq:expint}
\sup_{t \in [0,1]} \int_{\mathbb T^2} \exp(\beta|\nabla u(t,x)|)  \ dx < +\infty \text{ for some } \beta > 0.   
   \end{equation}
    Indeed, for this class of velocity fields, one can prove the analogue of the first statement of Lemma~\ref{lemma:pert_horseshoe_construction}, because for any $\eps > 0$, the rescaled velocity field $b_\eps (t,x) = \eps b(t,x/\eps)$ satisfies
\begin{align}
    \int_{B_{\eps}(0)} \exp(\beta |\nabla b_\eps(t,x)|) \ dx = \eps^2 \int_{B_{1}(0)} \exp(\beta|\nabla b(t,x)|)  \ dx,
\end{align}
which vanishes as $\eps \to 0$. Therefore, we can repeat the same construction in Example~\ref{example:vel_field}. 
A genericity result for such velocity fields, like the one stated for the Osgood case in Theorem~\ref{thm:main}, requires an adaptation of the proof presented in Section~\ref{sec:ProofOfMainThm}. Observe that if $u$ satisfies \eqref{eq:expint}, then its flow admits Sobolev estimates \cite{EBQHN21}, see \cite{MR4373167} for a related result on the vanishing viscosity limit in Yudovich's class.
\end{example}

\section{Proof of Theorem~\ref{thm:main}}\label{sec:ProofOfMainThm}

The next lemma shows that any $L^1_{\per} c_{\omega}$ velocity field can be perturbed so that it has a periodic point.

\begin{lemma}
\label{lemma:aux}
Let $\varepsilon > 0$ and $b \in L^1_{\per}(\R_+ ; c_{\omega}(\T^2; \R^2))$. Then, there exist $\tilde b \in L^1_{\per}(\R_+ ; c_{\omega}(\T^2; \R^2))$, $x \in \T^2$ and an integer $N \geq 1$ such that 
\begin{equation}\label{eq:ItIsAPertrubation}
    \| b-\tilde b\|_{L^1_{\per}c_{\omega}} < \varepsilon, \quad \text{and} \quad X_N^{\tilde{b}}(x) = x.
\end{equation}
\end{lemma}
 
\begin{proof}[Proof of Lemma~\ref{lemma:aux}]
If there exists some $x \in \T^2$ and $N \geq 1$ such that $X_N^b(x) = x$, the proof is finished. Therefore, we assume this is not the case.
Up to perturbing the velocity field, we can assume that $b$ is zero in the time interval $[1-\eta,1]$ for some small $\eta > 0$.
Let $\rho_0 > 0$ be a small number depending on $\eps$ and $\eta$ to be selected. \\
\begin{claim}
There exist $0 < \rho \leq \rho_0$, $x \in \T^2$ and an integer $N \geq 1$ such that 
\begin{equation}\label{eq:OutsideOfBalls}
 \metric(x, X_N^b(x)) \leq \rho, \quad \metric(X_{n}^{b}(x), x) \geq \rho , \quad \metric(X_{n}^{b}(x), X_{N}^{b}(x)) \geq \rho \quad \forall n = 1, \ldots, N-1.
 \end{equation}
\end{claim}
 \begin{proof}[Proof of Claim]
 Without loss of generality, up to choosing a smaller $\rho_0$, assume $\rho_0$ is smaller than the injectivity radius of $\T^2$.
 By compactness, there exists $y \in \T^2$ and $M \geq 1$ such that $\metric(y, X_M^b(y)) < \rho_0$. Take now $(\alpha,\beta) = \argmin_{\substack{0 \leq i < j \leq M}} \metric(X_{i}^b(y), X_j^b(y))$ and set $x \coloneqq X_\alpha^b(y)$ and $N \coloneqq \beta-\alpha$. Choose $\rho \coloneqq \metric(x, X_N^b(x))$, which makes the first inequality of \eqref{eq:OutsideOfBalls} true. Moreover, since $(\alpha,\beta)$ minimizes $(i,j) \mapsto \metric (X_i^b(y), X_j^b(y))$, the second and third conditions of \eqref{eq:OutsideOfBalls} hold. 
\end{proof}

Set $x^{\prime} = X_N^b(x)$ and note that due to the claim
\begin{equation}\label{eq:NoOtherPointsTouched}
    X_n^b(x) \notin B_{\rho}(x) \cup B_{\rho}(x^{\prime}) \supseteq B_{3\rho/4} \left( \frac{x + x^{\prime}}{2} \right) \quad \forall n = 1, \ldots, N-1.
\end{equation}
where $\frac{x + x^{\prime}}{2}$ denotes the midpoint of the unique minimizing geodesic between $x$ and $x^{\prime}$.
Due to this and the fact that the velocity field $b$ vanishes in the interval of times $[1-\eta,1]$, we can make $x$ a periodic point of another flow associated with $\tilde{b}$, a perturbation of $b$ in $[1-\eta,1] \times B_{3\rho/4} ( \frac{x + x^{\prime}}{2} )$.

 Let $\chi$ be a smooth cutoff such that $\chi = 1$ in $B_{\rho/2}(0)$ and $\chi = 0$ in $B_{3 \rho / 4}(0)^c$ and $\| \nabla \chi \|_{L^{\infty}} \leq C {\rho}^{-1}$.
 Take $\Omega \colon [0,1] \to \R$ such that $\supp \Omega \subseteq (1-\eta,1)$, $\| \Omega \|_{L^{\infty}} \leq C \eta^{-1}$ and $\int_0^1 \Omega \, dt = \pi$. Define $b_{\eps} \colon [0,1] \times \T^2 \to \R^2$ (which acts as a rotation by $\pi$ centered in $\frac{x + x^{\prime}}{2}$) by 
 \begin{equation}\label{eq:DefinitionOfPerturbationBloc}
    b_\varepsilon(t,y) \coloneqq \Omega(t) \chi \left(y-\frac{x+x'}{2}\right) J \left(y-\frac{x+x'}{2}\right), \quad J \coloneqq \begin{pmatrix}
    0 & -1 \\
    1 & 0
    \end{pmatrix}.
\end{equation}
Note that $\| b_{\eps}(t, \cdot) \|_{\omega} \leq C \eta^{-1} (\rho + \frac{\rho}{\omega(\rho)}) \leq \eps$ provided $\rho_0$ is sufficiently small since $\lim_{s \to 0} \frac{s}{\omega(s)} = 0$ and therefore by defining $\tilde{b} = b + b_{\eps}$, \eqref{eq:ItIsAPertrubation} holds. Since $\supp b_{\eps} \subseteq (1-\eta,1) \times B_{3 \rho/4} ( \frac{x + x^{\prime}}{2} )$ and due to \eqref{eq:NoOtherPointsTouched}, $X_t^b(x) = X_t^{\tilde{b}}(x)$ for all $t \in [0, N- \eta]$. Since $b_{\eps}$ generates a flow permuting $x$ and $X_N^b(x) = X_{N-\eta}^b(x)$, we obtain that $X_N^{\tilde{b}}(x) = x$. 
\end{proof}

\begin{lemma}\label{lemma:GoodScaleBadInterval}
    Let $b \in L^1_{\per}(\R_+; c_{\omega}(\T^2; \R^2))$ and $\eps > 0$ arbitrary. Then there exists $I = I(b,\eps) \subseteq [0,1]$ and $r = r(b, \eps) > 0$ such that
    \begin{equation}
        \sup_{\metric(x,y) < r} \dfrac{|b(t,x) - b(t,y)|}{\omega(\metric(x,y))} < \eps \quad \forall t \in I \quad \text{and} \quad \int_{[0,1] \setminus I} \| b(t, \cdot) \|_{C_{\omega}} \, dt < \eps.
    \end{equation}
\end{lemma}

\begin{proof}
    By uniform integrability, there exists $\delta> 0$ such that for any $I \subseteq [0,1]$ with $\Leb^1([0,1] \setminus I) < \delta$, we have $\int_{[0,1] \setminus I} \| b(t, \cdot) \|_{\omega} \, dt < \eps$. For any $t \in [0,1]$, $\lim_{r \to 0} \sup_{0 < \metric (x,y) < r} \frac{|b(t,x) - b(t,y)|}{\omega(\metric(x,y))} = 0$ and therefore there exists $r_t > 0$ such that $\sup_{0 < \metric (x,y) < r_t} \frac{|b(t,x) - b(t,y)|}{\omega(\metric(x,y))} < \eps$. The map $t \mapsto r_t$ can be taken measurable and therefore it suffices to take $I = \{ t \in [0,1] : r_t > r \}$ with $r$ small enough so that $\Leb^1([0,1] \setminus I) < \delta$.
\end{proof}

\begin{proof}[Proof of Theorem~\ref{thm:main}]
Since countable intersections of residual sets are residual, it suffices to show that for any $K \geq 1$, there exists a residual set $\mathcal{X}_K \subseteq L^1_{\per}(\R_+; c_{\omega}(\T^2; \R^2))$ such that for all $b \in \mathcal{X}_K$, we have $h_{\topo}(X_1^b) \geq K$. 
To prove this, we will show that for arbitrary $\eps > 0$ and $b \in L^1_{\per}(\R_+; c_{\omega}(\T^2; \R^2))$, there exist $\delta > 0$ and $v \in L^1_{\per}(\R_+; c_{\omega}(\T^2; \R^2))$ with $\| b - v \|_{L^1_{\per} c_{\omega}} < \eps$ such that for all $w \in L^1_{\per}(\R_+; c_{\omega}(\T^2; \R^2))$ with $\| w - v \|_{L^1_{\per} c_{\omega}} < \delta$ it holds that $h_{\topo}(X_1^w) \geq K$.

\textbf{Construction of $v$:} Without loss of generality, by Lemma~\ref{lemma:aux}, we may assume that $b$ has a periodic point $x \in \T^2$.
Let $N \geq 1$ be the period of $x$, i.e. the smallest positive integer such that $X_N^b(x) = x$.
Without loss of generality, up to translating the velocity field $b$ in $\T^2$, we may assume $x = 0$.
\begin{claim}\label{claim:FirstPerturbation}
    There exists $u \in L^1_{\per}(\R_+; c_{\omega}(\T^2; \R^2))$ and $r > 0$ such that $\| u - b\|_{L^1_{\per} c_{\omega}} < \eps / 2$ and 
    \begin{equation}\label{eq:b3def}
        u(t,y) = u(t,X_{t}^{b}(0)) \quad \forall y \in B_{r/2}(X_{t}^{b}(0)), \quad \forall t \in [0,N].
    \end{equation}
\end{claim}
\begin{proof}[Proof of Claim]
By Lemma~\ref{lemma:GoodScaleBadInterval}, there exist $I\subseteq [0,1]$ and $r > 0$ such that
\begin{equation}\label{eq:ApplicationOfLemmaGoodScaleBadInterval}
        \sup_{\metric(y,z) < r} \dfrac{|b(t,y) - b(t,z)|}{\omega(\metric(y,z))} < 10^{-3} \eps \quad \forall t \in I \quad \text{and} \quad \int_{[0,1] \setminus I} \| b(t, \cdot) \|_{C_{\omega}} \, dt < 10^{-3} \eps.
\end{equation}
Moreover, we choose $r$ small enough so that $\omega(r) < 10^{-2} \eps$ and such that the tubular sets 
\[
 T_i = \{ (t,y) \in [0,1] \times \T^2 : y \in B_r(X^b_{t+i}(0)) \}, \, i = 0, \ldots, N-1 \quad \text{are mutually disjoint.}
\]
We let $\chi$ be a smooth cutoff such that $\chi = 1$ in $B_{r/2}(0)$, $\chi = 0$ in $B_{r}(0)^c$, and $\| \nabla \chi \|_{L^{\infty}} \leq 9 r^{-1}$.
Define $u \in L^1_{\per}(\R_+; c_{\omega}(\T^2; \R^2))$ by
\begin{equation}
    u(t,y) =
    \left\{
        \begin{array}{ll}
        b(t,X^{b}_{t+i}(0)) \chi(y-X^{b}_{t+i}(0)) + 
        b(t,y) (1-\chi(y-X^{b}_{t+i}(0))) & \text{if $(t,y) \in T_i$;} \\
        b(t,y) & \text{otherwise.}
        \end{array}
    \right.
\end{equation}
To prove that $[ u(t, \cdot) - b(t, \cdot) ]_{\omega} < 10^{-1} \eps$, it suffices to verify that $[ u(t, \cdot) - b(t, \cdot) ]_{C_{\omega}(B_{r}(X^{b}_{t+i}(0)))} < 10^{-2} \eps$ since $u(t,\cdot) = b(t,\cdot)$ on $B_{r}(X^{b}_{t+i}(0))^c$.
A calculation shows that 
\begin{align*}
    (u-b)(t,y)-(u-b)(t,z) &= \Big( b(t,X^{b}_{t+i}(0)) - b (t,z) \Big)  \Big( \chi(y-X^{b}_{t+i}(0)) - \chi(z-X^{b}_{t+i}(0)) \Big) \\
    &\qquad + ( b(t,z) - b(t,y) ) \chi(y-X^{b}_{t+i}(0)) \quad \forall y,z \in B_{r}(X^{b}_{t+i}(0)).
\end{align*}
Hence, by Equation~\eqref{eq:ApplicationOfLemmaGoodScaleBadInterval}, for all $t \in I$ and all $y,z \in B_{r}(X^{b}_{t}(0))$,
\begin{align*}
    |(u-b)(t,y)-(u-b)(t,z)| &\leq 9 \cdot 10^{-3} r^{-1} \eps \omega(r) \metric(y,z) + 10^{-3} \eps \omega(\metric(y,z))  \leq 10^{-2} \eps \omega(\metric(y,z)),
\end{align*}
where we exploited the fact that $s \mapsto \omega(s)/s$ is decreasing
and thus $[ u(t, \cdot) - b(t, \cdot) ]_{{\omega}} < 10^{-1} \eps$ for all $t \in I$.
Moreover, it is clear that $\| u(t, \cdot) - b(t, \cdot) \|_{C^0} \leq \eps \omega (r) \leq 10^{-1} \eps$ for all $t \in I$ and therefore $\| u(t, \cdot) - b(t, \cdot) \|_{{\omega}} \leq \eps/4$ for all $t \in I$.
Following the same strategy, one can show that $\| u(t, \cdot) - b(t, \cdot) \|_{{\omega}} \leq 10^2 \| b(t, \cdot) \|_{{\omega}}$ for all $t \in [0,1] \setminus I$.
Therefore, using Equation ~\eqref{eq:ApplicationOfLemmaGoodScaleBadInterval},
\[
 \| u - b \|_{L^1_{\per} c_{\omega}} = \int_{I} \| u(t, \cdot) - b(t, \cdot) \|_{{\omega}} \, dt + \int_{[0,1] \setminus I} \| u(t, \cdot) - b(t, \cdot) \|_{{\omega}} \, dt < \frac{\eps}{4} + \frac{\eps}{10} < \frac{\eps}{2}.
\]
\end{proof}
Equation~\eqref{eq:b3def} implies $X_{t}^{u}(y) = X_{t}^{u}(0) + y$ for all $y \in B_{r/2}(0)$, i.e. locally in $B_{r/2}(0)$, the flow map of $u$ is a translation. 
Recall that $N$ is the period of $x = 0$.
By Lemma~\ref{lemma:pert_horseshoe_construction}, there exists a velocity field $h \colon [0,1] \times \T^2 \to \R^2$ such that $\| h \|_{L^1_{\per} c_{\omega}} < \eps/4$ and $\supp(h) \subseteq [0,1] \times B_{r/4}(0)$ with $h_{\topo}(X_1^h) > N K$. Let $v \in L^1_{\per}(\R_+; c_{\omega}(\T^2; \R^2))$ be defined on $[0,1] \times \T^2$ by
\[
 v(t,y) \coloneqq
 \left\{
    \begin{array}{ll}
     u(t,y) + \frac{1}{N} h\left(\dfrac{t+i}{N}, y - X_{t+i}^{u}(0)\right) &\text{if $(t,y) \in T_i$;} \\
     u(t,y) & \text{otherwise.}
    \end{array}
 \right.
\]
Note that this implies $v(t,y) = u(t,y) + N^{-1} h\left(t/N, y - X_{t}^{u}(0)\right)$ for all $t \in \R_+$ and all $y \in B_{r/4}(X^b_t(0))$. Indeed, all the functions involved are $N$-periodic, so it suffices to verify it for $t \in [0,N)$. Using that $y \in B_{r/4}(X_t(0))$ implies $(t - \lfloor t \rfloor, y) \in T_{\lfloor t \rfloor}$ for all $t \in [0,N)$, combined with the $1$-periodicity of $u$ and $v$, we get the identity. 
This expression for $v$ combined with Equation~\eqref{eq:b3def} implies that 
$$X_t^v (y) = X_t^{u}(0) + X_{\frac{t}{N}}^h(y) \quad \forall (t,y) \in \R_+ \times B_{r/4}(0)$$
and since $X_{N}^{u}(0) = 0$, we get $X_N^{v} = X_{1}^h$ in $B_{r/4}(0)$. From Lemma~\ref{lemma:pert_horseshoe_construction}, we deduce that $h_{\topo}(X_N^{v}) > NK$. In particular, $h_{\topo}(X_1^{v}) > K$ by Lemma~\ref{lemma:EntropyOfIterations}.
Finally, since $\| h \|_{L^1_{\per} c_{\omega}} <\eps/4$, we get $\| v - u \|_{L^1_{\per} c_{\omega}} < \eps/2$ and therefore $\| v - b \|_{L^1_{\per} c_{\omega}} < \eps$.
\\
\textbf{Lower bound on entropy in a neighborhood of $v$:}
Recall that $X_N^v = X_1^h$ in $B_{r/4}(0)$ and therefore Lemma~\ref{lemma:pert_horseshoe_construction} implies the existence of $\tilde{\delta} > 0$ such that if $\sup_{y \in B_{r/4}(0)} \metric( T(y) , X_N^v(y) ) < \tilde{\delta}$ then $h_{\topo}(T) > NK$.
Let $w \in L^1_{\per}(\R_+; c_{\omega}(\T^2; \R^2))$ be arbitrary and observe that, combining Equation~\eqref{eq:flowmap} for $X_t^v$ and $X^w_t$, adding and subtracting $v(t,X^w_t)$ and exploiting the Osgood regularity of $v$ and $w$, we have
\begin{equation}
\label{eq:integralBLS}
\metric(X^w_t(x), X^v_t(x)) \leq \int_0^t\omega(\metric(X^w_s(x), X^v_s(x))) \|v(s,\cdot)\|_\omega \ ds  + \int_0^t\|w(s,\cdot)-v(s,\cdot)\|_\omega \ ds.
\end{equation}
Finally, by the Bihari-LaSalle inequality, Equation~\eqref{eq:integralBLS} gives
\begin{equation}
    \metric(X_t^{w}(x), X_t^{v}(x)) \leq G^{-1} \left( G \left( \int_0^t \| w(s, \cdot) - v(s, \cdot) \|_{{\omega}} \, ds \right) + \int_0^t \| v(s, \cdot) \|_{{\omega}} \, ds \right) \quad \forall t \geq 0
\end{equation}
where $G\colon \R_+ \to \R$ denotes the function $G(s) = \int_1^{s} \frac{1}{\omega(\tau)} \, d \tau$. 

Therefore there exists a $\delta > 0$ such that if $\| w - v \|_{L^1_{\per} c_{\omega}} < \delta$ then $\sup_{y \in \T^2} \metric(X_N^{w}(y), X_N^{v}(y)) < \tilde{\delta}$ which leads to $h_{\topo}(X_N^w) > NK$.
Thus, due to Lemma~\ref{lemma:EntropyOfIterations}, $h_{\topo}(X_1^w) > K$ for all $w \in L^1_{\per}(\R_+; c_{\omega}(\T^2; \R^2))$ with $\| w - v \|_{L^1_{\per} c_{\omega}} < \delta$.
\end{proof}

\begin{remark}
\label{rmk:divfree}
When $b$ is divergence-free, $u$ in Claim~\ref{claim:FirstPerturbation} is in general not divergence-free. To construct a divergence-free $u$, one can carry out the perturbation at the level of the Hamiltonian. Indeed, if $\divergence b = 0$, then $b = a + \nabla^{\perp} H$ for some $a \in L^1([0,1];\R^2)$, $H \colon [0,1] \times \T^2 \to \R$ and it then suffices to perturb $H$.
\end{remark}

\bibliographystyle{plain}
\bibliography{biblio}
 
\end{document}